\documentclass{article}
\usepackage{fullpage,amsfonts,amsmath,hyperref,bm}
\usepackage{graphicx}
\newtheorem{theorem}{Theorem}

\newtheorem{lemma}[theorem]{Lemma}

\newenvironment{proof}[1][Proof]{\medskip\goodbreak\noindent\textbf{#1.} }{\ \rule{0.5em}{0.5em}}
\title{Separatrix splitting at  a Hamiltonian $0^2 i\omega$ bifurcation}
\author{V.Gelfreich, L.Lerman}
\author{Vassili Gelfreich${}^1$ and Lev Lerman${}^2$\\[6pt]
\small
${}^1$Mathematics Institute, University of Warwick, UK\\
\small
v.gelfreich@warwick.ac.uk\\[6pt]
\small
${}^2$Lobachevsky State University of Nizhni Novgorod, Russia\\\small  lermanl@mm.unn.ru
}
\date{9 September 2014}

\begin{document}
\maketitle
\begin{abstract}
We discuss the splitting of a separatrix in a generic unfolding of a degenerate equilibrium in a Hamiltonian system with two degrees of freedom.
We assume that the unperturbed fixed point has two purely imaginary eigenvalues and a double zero one. It is well known that an one-parametric
unfolding of the corresponding Hamiltonian can be described by an integrable normal form. The normal form has a normally elliptic invariant manifold
of dimension two. On this manifold, the truncated normal form has a separatrix loop. This loop shrinks to a point when the unfolding parameter vanishes.
Unlike the normal form, in the original system the stable and unstable trajectories of the equilibrium do not coincide in general. The splitting of
this loop is exponentially small compared to the small parameter. This phenomenon implies
non-existence of single-round homoclinic orbits and divergence of series in the normal form theory.
We derive an asymptotic expression for the separatrix splitting.
We also discuss  relations with behaviour of analytic continuation of the system in a complex neighbourhood of the  equilibrium.
\end{abstract}

\section{Set up of the problem}

Normal form theory provides a powerful tool for studying local dynamics near equilibria.
The normal form theory uses coordinate changes in order to represent equations in
the simplest possible form. The normal form often possesses additional symmetries which are not present in the original
system. For a Hamiltonian system a continuous family of symmetries implies existence of
an  additional integral of motion due to Noether theorem. In the case of two degrees of freedom,
an additional integral of motion makes the dynamics integrable. In this way the normal form
theory looses information on  non-integrable chaotic dynamics possibly present in the original system.

The accuracy of the normal
form theory depends on the smoothness of the original vector field and, in the analytic theory,
the error becomes smaller than any order of a small parameter
and in some cases an exponentially small upper bound can be established.

In this paper we illustrate this situation considering a classical generic bifurcation of an equilibrium in an one
parameter analytic family of Hamiltonian systems with two degrees of freedom
(see e.g.~\cite{BroerCKV1993}).
Let us describe our set up in more details.
Let $H_\mu=H_\mu(x_1,x_2,y_1,y_2)$
be a real-analytic family of Hamiltonian functions
defined in a neighbourhood of the origin in $\mathbb R^4$ endowed with the canonical symplectic form
$
\Omega=dx_1\wedge dy_1+dx_2\wedge dy_2
$.
The dynamics are defined via the canonical system
of Hamiltonian differential equations
$$
\dot x_k=\frac{\partial H_\mu}{\partial y_k}
\quad\mbox{and}\quad
\dot y_k=-\frac{\partial H_\mu}{\partial x_k}
$$
where $k\in\{1,2\}$. The corresponding flow preserves the Hamiltonian $H_\mu$
and symplectic form $\Omega$. It is convenient to write down the Hamiltonian equations in the vector form
\begin{equation}\label{Eq:mainHamEq}
\dot {\bm x}=\mathrm{J}H_\mu'({\bm x})
\end{equation}
where $\mathrm J$ is the standard symplectic matrix and $\bm x=(x_1,x_2,y_1,y_2)$.

We assume that for $\mu=0$ the origin is an equilibrium of the Hamiltonian system
with a pair of purely imaginary eigenvalues $\pm i\omega_0$ and a double zero one $\lambda_0=0$.
More precisely, we assume that $H_0'(0)=0$ and the Hessian matrix $H''_0(0)$ is
already transformed to the diagonal form $H_0''(0)=\mathrm{diag}(0,1,\omega_0,\omega_0)$,
which can be achieved by a linear canonical transformation provided
${\mathrm J}H_0''(0)$ is not semi-simple.
Then for any integer $n$ there is an analytic canonical change
of variables which transforms the Hamiltonian
to the following form
\begin{equation}\label{Eq:mainH}
H_\mu=\frac{y_1^2}2+V_\mu(x_1,I)+R_\mu(x_1,y_1,x_2,y_2)
\end{equation}
where $I=\frac{x_2^2+y_2^2}{2}$, $V_\mu(x_1,I)$ is polynomial in $x_1,I$ and $\mu$,
 and the remainder term $R_\mu$ has a Taylor expansion which starts with terms
 of order $n$, i.e., $R_\mu=O(\|\bm x\|^n+\mu^n)$. In the analytic case, this remainder can   be
made even exponentially small \cite{Iooss2004,Iooss2005,HaragusIooss2011}.
Our assumptions imply that
$$
\frac{\partial V_0}{\partial x_1}(0,0)=\frac{\partial^2 V_0}{\partial x_1^2}(0,0)=0.
$$
Then the lower order terms of $V_\mu$ have the form
\begin{equation}\label{Eq:potential}
V_\mu(x_1,I)=\omega_0I-a\mu x_1+b\frac{x_1^3}{3}+c x_1 I+\dots
\end{equation}
where we have explicitly written down  all quadratic and some of the cubic terms.
In a generic family $a,b,\omega_0\ne0$.
Without loosing in generality and for greater convenience we
assume $\omega_0,a,b>0$.

Since the remainder term in (\ref{Eq:mainH}) is small, it is natural to
make a comparison with the dynamics of the normal form
described by the truncated Hamiltonian function
\begin{equation}\label{Eq:HamNF}
\hat H_\mu=\frac{y_1^2}2+V_\mu(x_1,I).
\end{equation}
Obviously, the Poisson bracket $\{\hat H_\mu,I\}=0$ and consequently $I$
is an integral of motion for the normal form. For every fixed $I$ 
equation (\ref{Eq:HamNF}) represents 
a natural Hamiltonian system in a neighbourhood of the origin on the $(x_1,y_1)$-plane.
 Depending on the values of $\mu$ and $I$,  the potential $V_\mu$ takes one of three shapes shown on Figure~\ref{Fig:Vmu}.
\begin{figure}
\begin{center}
\includegraphics[width=0.30\textwidth]{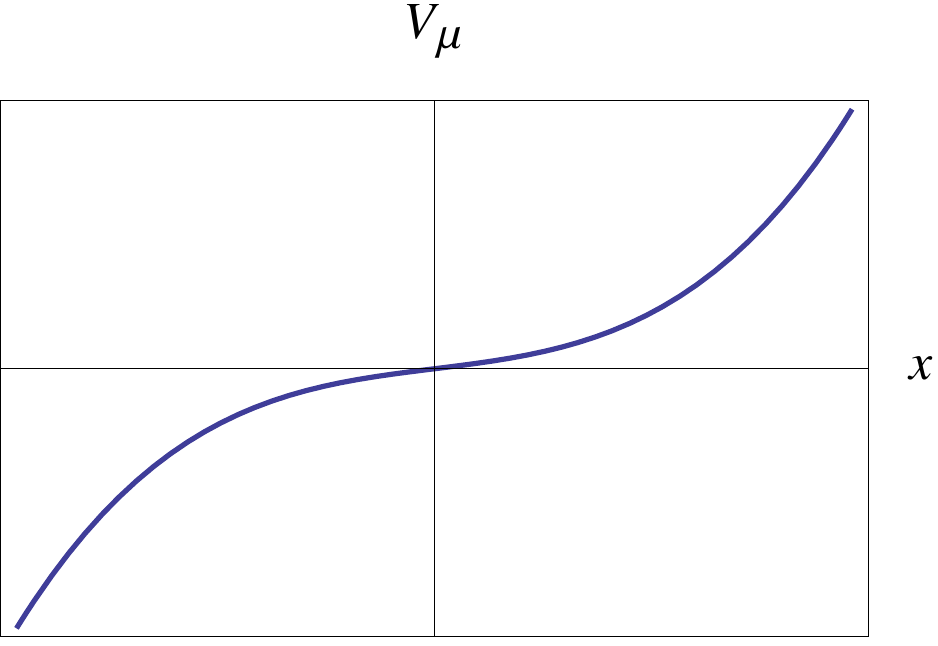}
~
\includegraphics[width=0.30\textwidth]{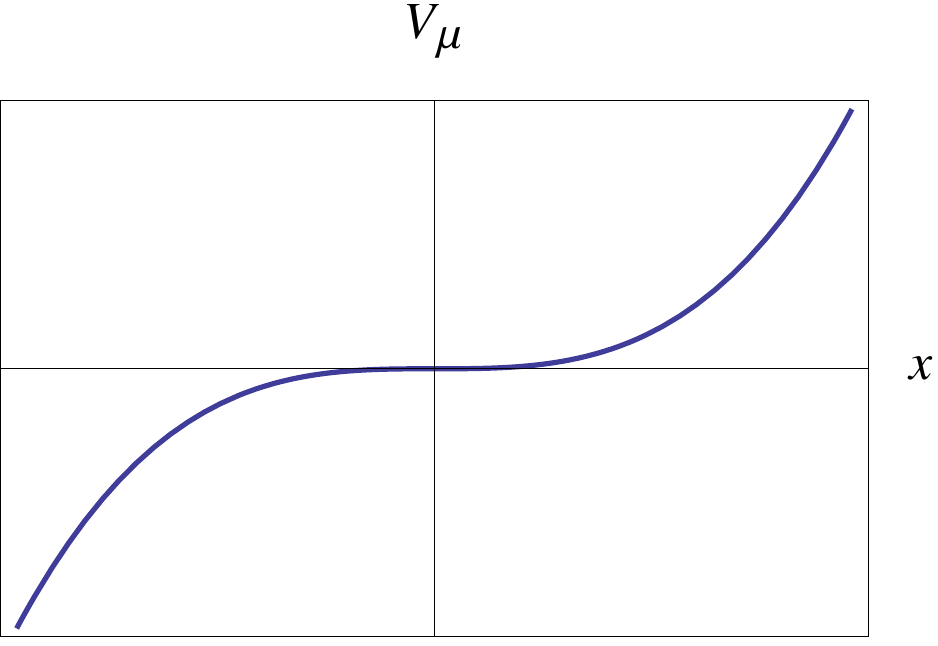}
~
\includegraphics[width=0.30\textwidth]{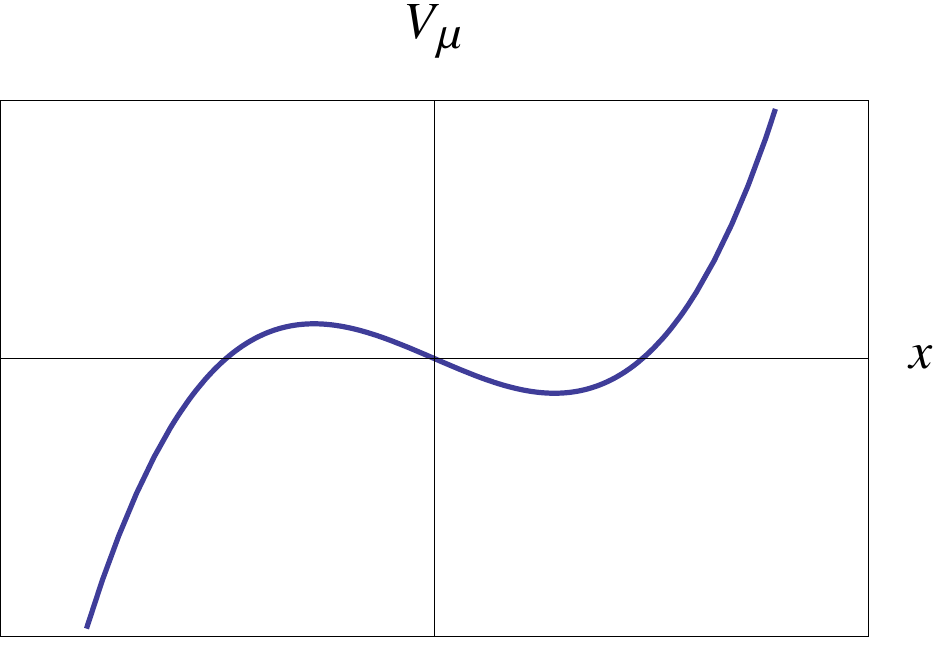}\\%
(a)\kern4.9cm (b)\kern4.9cm (c)~~
\end{center}
\caption{Potential of the truncated normal form for fixed values of $\mu$ and  $I$}
\label{Fig:Vmu}
\end{figure}
Respectively, the equation $\frac{\partial V_\mu}{\partial x_1}=0$ has either none, one
or two solutions located in a small neighbourhood of the origin.
These solutions correspond either to a periodic orbit (if $I>0$)
or to an equilibrium (if $I=0$) of the normal form.
We note that for $I=0$ Figure~\ref{Fig:Vmu} (a), (b) and (c) correspond respectively
to $\mu<0$, $\mu=0$ and $\mu>0$.
When the potential has the shape of Figure~\ref{Fig:Vmu}~(c) the corresponding Hamiltonian
system has a separatrix loop similar to the one shown on Figure~\ref{Fig:loop}.
\begin{figure}
\begin{center}
\includegraphics[width=4cm,height=2cm]{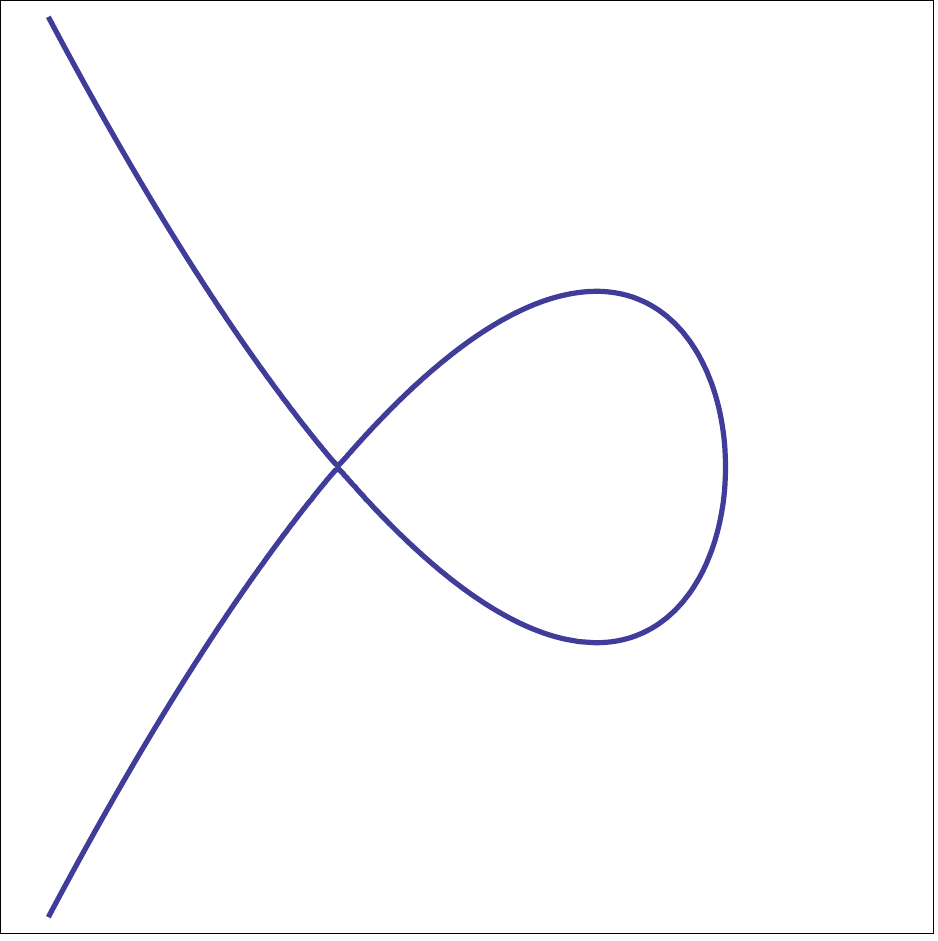}
\end{center}
\caption{Projection of the separatrix loop  on the $(x_1,y_1)$-plane}
\label{Fig:loop}
\end{figure}
This  separatrix loop  looks similar to the separatrix  of the equation
defined by the Hamiltonian
$$
\frac{y^2}{2}-a\mu x+b\frac{x^3}3.
$$
The situation can be summarised in the following way.
The plane $I=0$ is invariant for the normal form Hamiltonian $\hat{H}_\mu$. The restriction of the normal form Hamiltonian
onto this plane defines a Hamiltonian system with one degree of freedom.
As $\mu$ crosses the zero, a pair of equilibria is created on this plane, one saddle point
and one elliptic one. The normal form system has a separatrix solution which converges
to the saddle equilibrium both at $t\to+\infty$ and $t\to-\infty$. Trajectories located
inside this separatrix loop are periodic. All other trajectories escape
from a small neighbourhood of the origin and their behaviour cannot be studied
using only the local normal form theory presented here.

The remainder term in the Hamiltonian (\ref{Eq:mainH}) breaks the symmetry of the normal form
and it is expected that in general the full equations do not possess neither an
additional integral nor an invariant plane  \cite{GL2002,LermanG2005}.
Nevertheless, a part of the normal form dynamics
 survives. In particular, for $\mu>0$ the Hamiltonian system $H_\mu$ has a
 saddle-centre equilibrium
$\bm p_\mu$ with eigenvalues $\pm\lambda_\mu$ and $\pm i\omega_\mu$,
where $\omega_\mu$ is close to $\omega_0$ and $\lambda_\mu$ is of order of $\mu^{1/4}$.
There are 4 solutions (separatrices) of the Hamiltonian system which are asymptotic to this equilibrium.
These solutions converge to $\bm p_\mu$ as $t\to+\infty$ or $t\to-\infty$ and are
tangent asymptotically as $t\to \mp\infty$ to eigenvectors of $\mathrm JH_\mu''(\bm p_\mu)$, which correspond
to the eigenvalues $\pm\lambda_\mu$ respectively.

Two of these separatrices are close to the separatrix loop of the normal form.
The main objective of this paper is to study the difference between these separatrices.
Our main theorem implies that the unstable solutions returns to a small neighbourhood of $\bm p_\mu$ but,
in general, misses the stable direction by a quantity which is exponentially small compared to $\mu$.
Consequently, the system (\ref{Eq:mainH}) generically does not have a single-round homoclinic orbit
for all sufficiently small $\mu>0$.
We also point out that our Main theorem implies existence of homoclinic
trajectories for Lyapunov periodic orbits located on the central manifold
exponentially close to the saddle-centre equilibrium (compare with similar statements
for reversible systems in  \cite{Lombardi}
and with the recent preprint \cite{JNL2014}).

The dimension arguments \cite{KoltsovaL1995,Champneys2001}
show that the existence of a single-round homoclinic orbit
to a saddle-centre equilibrium of a vector field in $\mathbb R^4$
is expected to be a phenomenon of co-dimension between one and three
depending on the presence (or absence) of Hamiltonian structure and
reversible symmetries.  In particular, the codimension one corresponds to a symmetric homoclinic orbit for a symmetric equilibrium in
a reversible Hamiltonian system, and the codimension three corresponds to a non-symmetric homoclinic orbit
to a symmetric equilibrium in a reversible non-Hamiltonian system.
Treating $\lambda$ and $\omega$ as independent parameters,
Champneys  \cite{Champneys2001}  provided an example of 
 a reversible vector field where lines of homoclinic points bifurcate from $\lambda=0$ on the plane of~$(\lambda,\omega)$.

The splitting of the one-round separatrix loop does not prevent existence of ``multi-round" homoclinics,
i.e., homoclinic orbits which make several rounds close to the separatrix of the normal form
before converging to the equilibrium.
Generically, if the system is both Hamiltonian and reversible,
we expect the existence of reversible
multi-round homoclinics for a sequence of values of the parameter $\mu$
which converges to $0$.
The study of these  phenomena is beyond the goals of
this paper.

Since in the limit $\mu\to+0$ the separatrix loop disappears and the ratio $\omega_\mu/\lambda_\mu\to +\infty$,
the problem of the separatrix splitting near the bifurcation
can be attributed to
the class of singularly perturbed systems characterised by the presence of
two different time-scales, similar to the problems considered in
\cite{GL2002,GL2003,LermanG2005}.
The difficulty of a singularly perturbed problem is related to the
exponential smallness of the separatrix splitting in the parameter $\mu$
which requires  development of specially adapted perturbation methods
(see for example \cite{Treschev98,Lombardi,G2000} and references in the review \cite{GL01}).

The difficulties related to the exponential smallness do not appear in
problems of the regular perturbation theory.
At the same time dynamics of such systems share many
qualitative properties with the singularly perturbed case.

If a reversible Hamiltonian system has a symmetric
separatrix loop associated with a symmetric saddle-center
equilibrium, then its one-parameter reversible Hamiltonian
unfolding has multi-round homoclinic orbits for a set of
parameter values which accumulate at the critical one~\cite{Mielke,Ragazzo}.

A generic two parameter unfolding of a Hamiltonian system which has a homoclinic
orbit to a saddle-centre equilibrium was  studies in \cite{Koltsova},
where  countable sets of parameter values for
which 2-round (and multi-round) loops are found. 

The splitting of the separatrix loop has important consequences for the dynamics.
The problem of constructing a complete description of the dynamics in
a neighbourhood of a homoclinic loop to a saddle-centre was
stated and partially solved in \cite{Lerman1987}.
Later this result  was extended and improved 
in 
\cite{KoltsovaL1995,Ragazzo,Mielke}.

These papers do not directly cover the situation described in this paper
(see \cite{LermanG2005} for a discussion of  relations between these two classes
in the Hamiltonian context). The main difference is related to the exponential smallness
of the separatrix splitting in the bifurcation problem discussed in the present paper.
The presence of exponentially small phenomena hidden beyond all orders
of the normal form theory is also observed in other bifurcation problems
(see for example \cite{Lombardi,Gel00,Gel02,BaldomaSeara}).

Finally, we note that the problem of existence of small amplitude single- and multi-round homoclinic orbits
arises in various applications. These applications include dynamics of the
three-body-problem near $L_2$ libration point \cite{Simo}.
Homoclinic solutions also appear in the study of traveling wave or steady-state reductions of partial differential equations
on the real line which model various phenomena in mechanics, fluids and optics
(for more details see \cite{Champneys1998,Champneys2001}). These solutions are of particular interest as they represent localized modes or solitary waves,
these problems are often of the singular perturbed nature~\cite{AEK}.

\section{Symplectic approach to measuring the separatrix splitting}

Let $\bm p_\mu$ be an equilibrium  of the Hamiltonian system with
eigenvalues $(\pm\lambda_\mu,\pm i\omega_\mu)$. Then for each $\mu\in(0,\mu_0)$
(where $\mu_0$ is a positive constant) there is an analytic change of variables such that
the equilibrium is shifted to the origin and the Hamiltonian function is transformed to its 
Birkhoff normal form which can be presented in the form 
\begin{equation}\label{Eq:hamNFl}
H_\mu=h_\mu(E_h,E_e)
\end{equation}
where $h_\mu$ is an analytic function of two variables
\begin{equation}\label{Eq:EeEh}
E_h=x_1y_1
\qquad\mbox{and}\qquad
E_{e}=\frac{x_2^2+y_2^2}2.
\end{equation}
A statement equivalent to the convergence of the normal form
was originally obtained in \cite{Moser1958}.
Of course, since $\lambda_\mu\to0$ as $\mu\to+0$,  the size of domains of convergence shrinks to zero both for the normal form
and for the normalising transformation. Nevertheless, it is possible to refine
the estimates of  \cite{Giorgilli2001} in order to establish that the sizes of the domains are
sufficiently large to be used in the following arguments.

Obviously
$\{E_h,h_\mu\}=\{E_e,h_\mu\}=0$ and, consequently, both
$E_h$ and $E_e$ are constant along trajectories of the Hamiltonian system.
Both functions are local integrals only and in general do not have a single-valued
extension onto the phase space. The transformation which transforms the original Hamiltonian to
the normal form is not unique. Nevertheless the values of $E_h$ and $E_u$ are unique
as they do not depend on this freedom.

As the eigenvalues of the equilibrium are preserved, the Taylor expansion of the transformed
Hamiltonian has the form
\begin{equation}\label{Eq:quadr}
h_\mu(E_h,E_e)=H_\mu(\bm p_\mu)+\lambda_\mu E_h+\omega_\mu E_e+O(E_e^2+E_h^2).
\end{equation}
The structure of the phase space in a neighbourhood of the origin is illustrated
by Figure~\ref{Fig:sc} where $H=H_\mu-H_\mu(\bm p_\mu)$.
In particular, in the normal form coordinates points with $x_1=y_1=0$ 
correspond to Lyapunoff periodic orbits.
\begin{figure}
\begin{center}
{\includegraphics[trim= 5cm 5cm 10cm 7cm,angle=-90,width=0.8\textwidth]{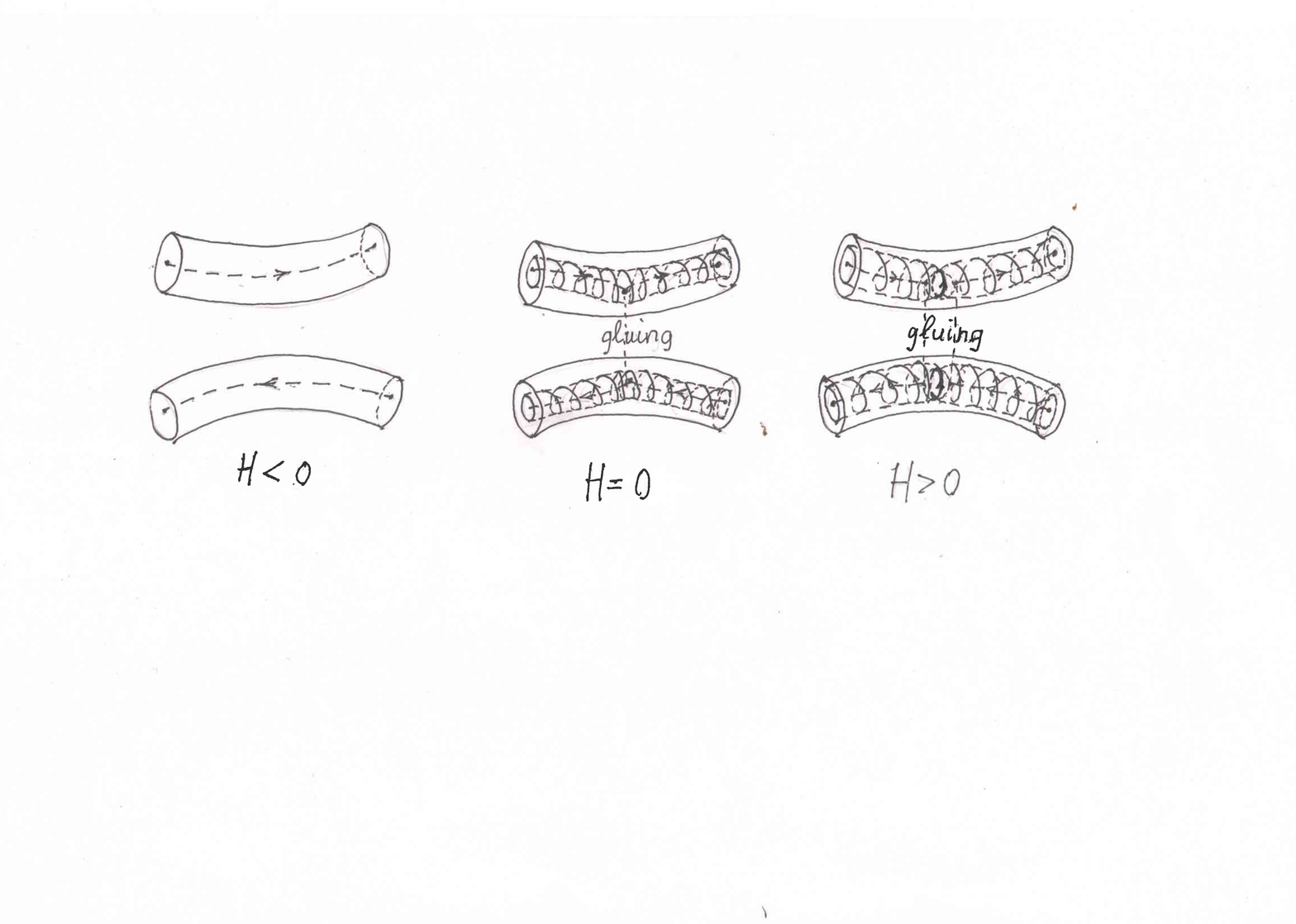}}
\end{center}
\caption{Structure of the phase space in a neighbourhood of a saddle-centre}
\label{Fig:sc}
\end{figure}
A  trajectory which converges to $\bm p_\mu$ as $t\to-\infty$ or $t\to\infty$ without leaving the domain of the normal form
 has $E_e=E_h=0$. In the normal form coordinates all these trajectories are easy to find explicitly.

Let $\bm x_\mu^\pm(t)$ be separatrix solutions of the Hamiltonian system
(\ref{Eq:mainH}) which converge to the equilibrium
\begin{equation}\label{Eq:hameq}
 \lim_{t\to-\infty} x_\mu^{-}(t)=\bm p_\mu
\qquad\mbox{and}\qquad
 \lim_{t\to+\infty} x_\mu^{+}(t)=\bm p_\mu
\end{equation}
being close to the separatrix loop described in the introduction.
Since $\bm x_\mu^\pm$ are solutions of an autonomous ODE,
these assumptions define the functions $\bm x_\mu^\pm(t)$
up to a translation in time $t$. We will eliminate this freedom later.
At the moment it is sufficient to note that
$\bm x_\mu^\pm(0)$ will be chosen to be in a small neighbourhood of the intersection of the normal
form separatrix with the plane
$\Sigma=\{\,y_1=0\,\}$. Note that the curve $\bm x_\mu^\pm$
may have more than one intersection with $\Sigma$, in this case we chose
a ``primary" one.

The unstable separatrix $\bm x_\mu^-(t)$  leaves the domain of the normal form, makes a round trip
near the ghost separatrix loop, and at a later moment of time comes back
close to the stable direction of the Hamiltonian vector field at $\bm p_\mu$.
Let $E_e^1$ and $E_h^1$ be the values of the elliptic and hyperbolic energies obtained after this round-trip.
Conservation of the energy implies that $h_\mu(E_h^1,E_e^1)=h_\mu(0,0)$, so the values of $E_e^1$ and $E_h^1$
are not independent. Traditionally the elliptic energy $E_e^1$ is used to measure the separatrix
splitting. In particular, if $E^1_e=0$, then the trajectory is homoclinic. If $E^1_e\ne0$, the trajectory
will eventually leave the neighbourhood of $\bm p_\mu$ for the second time.

\begin{theorem}[Main theorem]
There is a sequence of real constants $(a_k)_{k\ge0}$ such that
\begin{equation}\label{Eq:mainasymp}
E_e^1\asymp e^{-2\pi\omega_\mu/\lambda_\mu}\left(a_0+\sum_{k\ge2}a_k\lambda_\mu^{2k}\right)\,.
\end{equation}
The coefficient $a_0=|b_0|^2/2$, where $b_0$ is a complex constant defined by the Hamiltonian $H_0$ via equation\/~{\rm (\ref{Eq:a0})}
and $a_n$ are defined by\/~{\rm (\ref{Eq:an})}.
\end{theorem}

We note that $\lambda_\mu\sim\mu^{1/4}$ and $\omega_\mu=\omega_0+O(\mu^{1/2})$. Then
the asymptotic expansion  (\ref{Eq:mainasymp}) implies that $E_e^1$ is exponentially small compared
to $\mu$.
Moreover, if $a_0\ne0$ this theorem implies the splitting of the separatrix and, hence, non-existence of a single-loop
homoclinic orbit. We do not know an explicit formula to compute $b_0$.
Nevertheless, numerical methods of \cite{GS2008} can be adapted for evaluating the constants
in the asymptotic series with arbitrary precision. The arguments presented in section~\ref{Se:melnikov}
can be used to prove that
$b_0$ is generically non-vanishing (as the map $H_0\mapsto b_0$  is a non-trivial analytic (non-linear) functional).
Indeed, if the Hamiltonian $H_\mu$ analytically depends on an additional parameter $\nu$, then
it can be proved that $b_0=b_0(\nu)$ is analytic.
Then the Melnikov method can be used to
show that $b_0'\ne0$ for values of $\nu$ which correspond to an integrable Hamiltonian.
Finally the analyticity implies that zeroes of $b_0(\nu)$ are isolated and, consequently,
the coefficient does not vanish for a generic Hamiltonian $H_0$.

The proof of the main theorem is based on ideas proposed by V. Lazutkin
in 1984 for  studying  separatrix splitting for the standard map and later
used in  \cite{G2000} for studying separatrix splitting of a rapidly forced pendulum.
This paper contains a sketch of the proof for the main theorem.

The Melnikov method is often used to study the splitting of separatrices. In general the Melnikov
method does not produce a correct estimate for the problem discussed in this paper.
Section \ref{Se:melnikov} contains a discussion of the applicability of the Melnikov method.

\section{Elliptic energy and the variational equation}

As a first step of the proof we provide a description of the elliptic energy $E_e^1$ in terms
of the splitting vector
\begin{equation}\label{Eq:delta}
\bm\delta_\mu(t)=\bm x_\mu^+(t)-\bm x_\mu^-(t),
\end{equation}
which describes the difference between the stable and unstable separatrix solution,
and a solution of a variational equation around $\bm x_\mu^+(t)$.
This description allows us to compute $E^1_e$ without
 explicit usage of a transformation to the normal form in a neighbourhood of the saddle-centre equilibrium.

In the normal form coordinates the Hamiltonian is described by equation (\ref{Eq:hamNFl}),
thus the corresponding equations of motion take the form
\begin{equation}\label{Eq:nfhameq}
\dot x_1=\tilde\lambda x_1,\quad
\dot y_1=-\tilde\lambda y_1,\quad
\dot x_2=\tilde\omega  x_2,\quad
\dot y_2=-\tilde\omega  y_2,
\end{equation}
where $\tilde\lambda=\partial_1h_\mu(E_h,E_e)$ and
$\tilde\omega=\partial_2h_\mu(E_h,E_e)$.
Since on the local stable trajectory  $E_h=E_u=0$, equation~(\ref{Eq:quadr}) implies that $\tilde\lambda(0,0)=\lambda_\mu$,
and we can find this trajectory explicitly:
\begin{equation}\label{Eq:xplusNF}
\bm x_\mu^+(t)=(0,c_\mu e^{-\lambda_\mu t},0,0)
\end{equation}
where $c_\mu$ is a constant.
Then the  variational equation around this solution takes the form
$$
\dot \xi_1 
=\lambda_\mu \xi_1
\quad
\dot \eta_1 = -\lambda _\mu \eta_1
-\kappa_\mu y_1^2(t)\xi_1 ,
\qquad
\dot \xi_2
=\omega_\mu \eta_2,
\qquad
\dot \eta_2
=-\omega_\mu \xi_2
$$
where $y_1(t)=c_\mu e^{-\lambda_\mu t}$ and $\kappa_\mu=\partial^2_{1,1}h_\mu(0,0) $.
A fundamental system of solutions for the variational equation
is found explicitly:
\begin{equation}\label{Eq:varNF}
\begin{array}{ll}
\bm\xi_1(t)=e^{-i\omega_\mu t}(0,0,1,-i),\quad&
\bm\xi_2(t)=e^{i\omega_\mu t}(0,0,1,i),\\[8pt]
\bm\xi_3(t)=\lambda_\mu c_\mu e^{-\lambda_\mu t}(0,1,0,0),\qquad&
\bm\xi_4(t)=-\lambda_\mu^{-1} c_\mu^{-1} (e^{\lambda_\mu t},c_\mu^2\kappa_\mu te^{-\lambda_\mu  t},0,0).
\end{array}
\end{equation}
Note that we have chosen $\bm\xi_3(t)=\dot{\bm x}_\mu^+(t)$.
The first two solutions are  mutually complex conjugate,
so
 real-analytic solutions can be easily constructed when needed.
A direct computation shows that $\Omega({\bm\xi}_1,{\bm\xi}_2)=2i$ and
$\Omega(\bm\xi_3,\bm\xi_4)=1$. For all other pairs $(j,k)$ the symplectic form
$\Omega(\bm\xi_j,\bm\xi_k)$  vanishes. Later we will study those solutions for
non-real values of $t$. So it is interesting to note that the function $\bm\xi_1(t)$
exponentially grows in the complex upper half-plane $\Im t>0$, while $\bm\xi_2(t)$
exponentially decays there.

For each fixed value of $t$ we can consider the collection of vectors $\bm\xi_k(t)$, $k=1,\dots,4$, as a basis in $\mathbb C^4$.
Then the function
\begin{equation}\label{Eq:ttheta}
\tilde\theta_1(t):=\frac{1}{2i}\Omega(\bm\delta_\mu(t),\bm\xi_2(t))
\end{equation}
provides the $\bm\xi_1$-component of the splitting vector $\bm\delta_\mu(t)$.
Equations (\ref{Eq:varNF}) and (\ref{Eq:xplusNF}) imply that
$$
\tilde\theta_1(t)=\frac{1}{2i}e^{i\omega_\mu t}(ix^-_2(t)-y^-_2(t))\,,
$$
where $x_2^-$ and $y_2^-$ are components of $\bm x^-_\mu$ in the normal
form coordinates (for the values of $t$ corresponding to the first return of the unstable trajectory
to the small neighbourhood of the saddle-centre $\bm p_\mu$). Taking into account that $\bm x^-_\mu$ is real-analytic
and using the definition of $E_e^1$ of (\ref{Eq:EeEh}),
we obtain
that the equality
\begin{equation}\label{Eq:Eellip}
|\tilde\theta_1(t)|^2=\frac{(x_2^-(t))^2+(y_2^-(t))^2}{4}=
\frac{E_e^1}{2}.
\end{equation}
holds for real values of $t$.  Since $E_e^1$ is a local integral,
$|\tilde\theta_1(t)|$ also stays constant for real $t$ while the unstable
solution remains inside the domain of the normal form.

The equation (\ref{Eq:Eellip}) provides a relation between $E_e^1$ and $\tilde\theta_1$.
While $E_e^1$ is defined using the normal form coordinates, the function $\tilde\theta_1$
is defined by (\ref{Eq:ttheta}) and can be evaluated in other canonical systems of coordinates.
This computation relies on accurate analysis of the way the splitting vector $\bm\delta_\mu$
and the solutions of the variational  equation $\bm\xi_k$ are transformed under
coordinate changes.
It is important to note that although canonical coordinate changes do preserve the symplectic
form,  $\Omega(\bm\delta_\mu(t),\bm\xi_2(t))$ does not take the same value when
evaluated in a different coordinate system but can differ by a value of the order of $\|\bm\delta_\mu(t)\|^2$.

Slightly overloading the notation, let $\bm\delta_\mu$ and $\bm\xi_k$ be respectively the splitting vector and the solutions of the variational
equation written in the original coordinates.
The splitting vector is defined by equation (\ref{Eq:delta}).
The solutions $\bm\xi_k$ can be fixed by asymptotic conditions described in the next section
to ensure that they represent  the same solutions of the variational equation as in (\ref{Eq:ttheta}) but
expressed in the other coordinates.
Then we define a function in a way similar to (\ref{Eq:ttheta})
\begin{equation}\label{Eq:theta1}
\theta_1(t)=\frac1{2i}\Omega(\bm \delta_\mu(t),\bm \xi_2(t)).
\end{equation}
It is easy to check that
$\tilde\theta_1(t)=\theta_1(t)+O(C_\mu\|\delta_\mu(t)\|^2)$, where the constant $C_\mu$
bounds the $C^2$-norm of the transformation between the systems of coordinates.
For the real $t$, the function $\bm\xi_2(t)$ is uniformly bounded and we conclude that
$|\theta_1(t)|=O(\|\delta_\mu(t)\|)$. We conclude that $|\tilde\theta_1(t)|^2=|\theta_1(t)|^2+O(C_\mu\|\delta_\mu(t)\|^3)$.
Then  equation  (\ref{Eq:Eellip}) implies that
\begin{equation}\label{Eq:EellipGen}
E_e^1=2|\theta_1(t)|^2+O(C_\mu\|\delta_\mu(t)\|^3)\,.
\end{equation}
A refinement of the arguments from \cite{Giorgilli2001} implies that $C_\mu=O(\mu^{-2})$.
This factor  does not break the approximation as $\theta_1$ is of the same order as $\delta_\mu$ and $\delta_\mu$
is exponentially small compared to $\mu$. We will use the equation (\ref{Eq:EellipGen}) to obtain an estimate
for $E_e^1$.

\section{Variational equation\label{Se:vareq}}
On the next step of the proof we study solutions of the variational equation near the
unstable separatrix solution  $\bm x_\mu^-$:
\begin{equation}\label{Eq:mainvareq}
\dot{\bm \xi}=\left.\mathrm{J}H_\mu''\right|_{\bm x_\mu^-(t)}\bm \xi\,.
\end{equation}
This is a linear homogeneous non-autonomous equation.
Since the variational equation comes from a Hamiltonian system, it is easy to check that
for any two solutions $\bm\xi$ and $\tilde{\bm\xi}$ of equation (\ref{Eq:mainvareq}) the
value of the symplectic form $\Omega(\bm\xi(t),\tilde{\bm\xi}(t))$ is independent of $t$.
This property together with  asymptotic behaviour of solutions at $t\to-\infty$
are  used to select a fundamental system of solutions.

Let $\bm v_\mu\in\mathbb C^4$ be an eigenvector of the linearised Hamiltonian vector field at $\bm p_\mu$,
$$
\left.\mathrm{J}H_\mu''\right|_{\bm p_\mu}\bm v_\mu=i\omega_\mu \bm v_\mu,
$$
such that  $\Omega(\bm v_\mu,\bar{\bm  v}_\mu)=-2i$.
Note that the complex conjugate vector $\bar {\bm v}_\mu$ is also an eigenvector, but it
corresponds to the complex conjugate eigenvalue $-i\omega_\mu$. Let
$\bm v_\mu=\bm v_\mu'+i \bm v_\mu''$ with $\bm v_\mu',\bm v_\mu''$ being vectors with real components.
Then our normalisation condition is equivalent to $\Omega(\bm v_\mu',\bm v_\mu'')=1$.
The vector $\bm v_\mu$ is  defined uniquely up to multiplication by a complex constant of unit absolute value, in other words,
 for any real $c$ the vector  $e^{i c}\bm v_\mu$ also satisfies our normalisation assumption.
 We assume that this freedom is eliminated in the same way as in the linear part of
 the normal form theory near the saddle-centre. In particular,
 $v_\mu$ is a smooth function of $\mu^{1/2}$ (including the limit $\mu\to+0$).

Now we are ready to define fundamental solutions of the variational equation.
One solution is selected by the assumption
$$
\bm\xi_2(t)=e^{i\omega_\mu t}\bm v_\mu+O(e^{\lambda_\mu t})
$$
for $t\to-\infty$. 
The other one  is defined using the real symmetry:
$$
\bm\xi_1(t)=\overline{{\bm\xi}_2(\bar{t})}\,.
$$
These two solutions are not real on the real axis and $\Omega(\bm \xi_1,\bm \xi_2)=2i$.
Sometimes it is useful to consider their linear combinations
$$
\tilde{\bm \xi}_1(t)=\frac{\bm  \xi_1(t)+ \bm \xi_2(t)}{2}
\qquad\mbox{and}\qquad
\tilde{\bm  \xi}_2(t)=\frac{\bm \xi_1(t)-\bm \xi_2(t)}{2i}
 $$
 which are  real-analytic.

The third solution is given by
$$
\bm\xi_3(t)=\dot{\bm x}^-_\mu(t)
\,.
$$
The last solution is chosen to satisfy the following normalisation conditions:
$$
\Omega(\bm\xi_{1},\bm\xi_4)=\Omega(\bm\xi_{2},\bm\xi_4)=0,\qquad \Omega(\bm\xi_3,\bm\xi_4)=1
.
$$
We note that since the original system is Hamiltonian,  for any two functions  $\bm \xi$ and $\tilde {\bm \xi}$
which satisfy the variational equation,  $\Omega(\bm \xi(t),\tilde {\bm \xi}(t))$
is independent of $t$.
Consequently the vectors $\tilde{\bm \xi}_1(t)$, $\tilde{\bm \xi}_2(t)$, $\bm \xi_3(t)$, $\bm \xi_4(t)$ form
a standard symplectic basis  for every~$t$:
\begin{equation}\label{Eq:normsymp}
\Omega(\tilde{\bm \xi}_1,\tilde{\bm \xi}_2)=
\Omega(\bm \xi_3,\bm \xi_4)=1,\qquad
\Omega(\bm \xi_1,\bm \xi_3)=\Omega(\bm \xi_2,\bm \xi_3)=\Omega(\bm \xi_1,\bm \xi_4)=\Omega(\bm \xi_2,\bm \xi_4)=0.
\end{equation}
Then we can write the splitting vector in this basis:
$$
\bm\delta_\mu(t)=\tilde\theta_1(t)\tilde{\bm\xi}_1(t)
+\tilde\theta_2(t)\tilde{\bm\xi}_2(t)+
\theta_3(t)\bm\xi_3(t)+
\theta_4(t)\bm\xi_4(t).
$$
The normalisation condition (\ref{Eq:normsymp}) implies that
$$
\tilde\theta_1=\Omega(\bm\delta_\mu,\tilde{\bm\xi}_2),\quad
\tilde\theta_2=-\Omega(\bm\delta_\mu,\tilde{\bm\xi}_1),\quad
\theta_3=\Omega(\bm\delta_\mu,\bm\xi_4),\quad
\theta_4=-\Omega(\bm\delta_\mu,\bm\xi_3).
$$
For the future use we also define
$$
\theta_1=\Omega(\bm\delta_\mu,\bm\xi_2),\quad
\theta_2=\Omega(\bm\delta_\mu,\bm\xi_1),
$$
which involve the non-real solutions of the variational equation.
The real symmetry implies that $\theta_2(t)=\overline{\theta_1(\overline{t})}$.

We note that in general the coefficients $\theta_k$ depend on time.
Indeed, the equation (\ref{Eq:mainHamEq}) implies that
$$
\dot{\bm\delta}_\mu=\mathrm{J}(H'(\bm x^-_\mu+\bm\delta_\mu)-H'(\bm x^-_\mu))=\left.\mathrm{J}H''\right|_{\bm x^-_\mu}\bm\delta_\mu
+{\bm F}_2(\bm x^-_\mu,\bm \delta_\mu)
$$
where $\bm F_2(\bm x^-_\mu,\bm \delta_\mu)=O(\|\bm \delta_\mu\|^2)$  is a remainder of a Taylor series.
Then differentiating the definition of $\theta_k$ with respect to $t$
and taking into account that $\left.\mathrm{J}H''\right|_{\bm x^-_\mu}$ is a symplectic matrix
we get
$$
\dot\theta_k=\pm\Omega(\bm F_2,\bm \xi_{k'(k)})=O(\|\bm \delta_\mu\|^2)\,,
$$
where $k'(k)$ is the index of the canonically conjugate variable (e.g. $k'(1)=2$ and $k'(2)=1$).
So $\theta_k$ are $\|\bm\delta_\mu\|^2$-close to  being constant. Moreover,  $\theta_3,\theta_4$ are much smaller than $\theta_1$
and $\theta_2$.
Indeed, taking into account the definition of $\bm\xi_3$ we get
$$
\theta_4=-\Omega(\bm\delta_\mu,\bm\xi_3)=-\Omega(\bm\delta_\mu,\dot{\bm x}^-_\mu)=d_{\bm x^-_\mu} H_\mu(\bm\delta_\mu),
$$
where $d_{\bm x^-_\mu} H_\mu$ is the differential of $H_\mu$ at the point ${\bm x^-_\mu}$.
Since $H(\bm x^-_\mu+\bm\delta)=H(\bm x^+_\mu)=H(\bm x^-_\mu)$, we conclude that $\theta_4=O(\|\bm\delta_\mu\|^2)$.

Initial condition for $\bm x_\mu^\pm(0)$ can be chosen in such a way that $\Omega(\bm\delta_\mu(0),\bm\xi_3(0))=0$
(by translating time in the stable solution in order to achieve the zero
projection of $\bm\delta_\mu(0)$ on the direction of the Hamiltonian vector field 
at $\bm x^-(0)$ represented by $\bm\xi_3(0)$).
Thus $\theta_3=O(\|\bm\delta\|^2)$.


Taking into account the real symmetry we see that
the problem of the separatrix splitting is reduced to the study of a single complex constant $\theta_1(0)$
(via the equation (\ref{Eq:EellipGen})).

\section{Formal expansions\label{Se:formal}}

The proof of the main theorem requires construction of accurate approximations for the
stable and unstable separatrix solutions of the Hamiltonian system (\ref{Eq:mainHamEq})
as well as the fundamental solutions of the variational equation (\ref{Eq:mainvareq}).
Taking into account that the Hamiltonian (\ref{Eq:mainH}) can be formally transformed
to the integrable normal form (\ref{Eq:HamNF}), we construct an approximation by
finding a formal solution to the systems defined by the normal form Hamiltonian.
Of course, the series of the normal form theory diverge in general, but
they can be shown to provide asymptotic expansions for the true solutions
restricted to properly chosen domains on the complex plane of the time variable $t$.

\subsection{Formal separatrix}
In this section we find a formal separatrix  for  the normal form Hamiltonian
\begin{equation}\label{Eq:Hnf}
\hat H_\mu=\frac{y_1^2}2+\hat V_\mu(x_1,I)
\end{equation}
where  $\hat V_\mu$ is a formal series in three variables $x_1,I$ and $\mu$ with
the lower order terms given by (\ref{Eq:potential}).
The corresponding Hamiltonian system has the form
\begin{equation}\label{Eq:formalHamEq}
\begin{array}{ll}
\dot x_1=y_1,\quad
&\dot x_2=\partial_2\hat V_\mu( x_1,I)y_2,\\[8pt]
\dot y_1=-\partial_1\hat V_\mu(x_1,I),\quad
&\dot y_2=-\partial_2\hat V_\mu(x_1,I)x_2.
\end{array}
\end{equation}
Obviously,  the plane $x_2=y_2=0$ is invariant
and we construct  a formal separatrix located on this plane.
Formal expansions can be substantially simplified
with the help of an auxiliary small parameter~$\varepsilon$.
So instead of performing expansion directly in powers of $\mu$,
we look for a solution of the system (\ref{Eq:formalHamEq}) considering
$x_1,$ $y_1$ and $\mu$  as formal power series in $\varepsilon$.
The coefficients of the series for $x_1$ and $y_1$
are assumed to be functions of the  slow time $s=\varepsilon t$.
The following lemma establishes existence and uniqueness of
a formal solution in a specially designed class of formal series.

It is important to note that the leading terms in the series $\hat x_1$ and $\hat y_1$
are of the form $\varepsilon^2 p_1(s)$ and $\varepsilon^3 q_1(s)$ respectively.
This choice makes the expansions of Lemma~\ref{Le:fsep}
compatible with approximations for the separatrix obtained using the standard scaling,
a traditional tool used in the bifurcation theory.

\begin{lemma}\label{Le:fsep}
There are unique real coefficients $p_{k,l}$, $q_{k,l}$ and $\mu_k$
such that the formal series
\begin{equation}\label{Eq:fomsep}
\hat x_1=\sum_{k\ge 1}\varepsilon^{2k}p_k(\varepsilon t),\qquad
\hat y_1=\sum_{k\ge 1}\varepsilon^{2k+1}q_k(\varepsilon t),\qquad
\mu=\sum_{k\ge2}\mu_k\varepsilon^{2k},
\end{equation}
where $p_1(s)$ is not constant and the coefficients $p_k$, $q_k$
with $k\ge1$ have the form
\begin{equation}\label{Eq:pkqk}
p_k(s)=\sum_{l=0}^{k}\frac{p_{k,l}}{\cosh^{2l}\tfrac s2},\qquad
q_k(s)=\sum_{l=1}^{k}\frac{q_{k,l}\sinh\tfrac s2}{\cosh^{2l+1}\tfrac s2},
\end{equation}
 together with $x_2=y_2=0$ satisfy the Hamiltonian equations\/
{\rm  (\ref{Eq:formalHamEq})}.
Moreover, if $\hat H_\mu$ is a formal normal form for the analytic family $H_\mu$ defined by\/ {\rm (\ref{Eq:mainH})},
then the series for $\mu$ are convergent and~$\varepsilon=\lambda_\mu$.
\end{lemma}

\begin{proof}
The restriction of the system (\ref{Eq:formalHamEq})  onto
the invariant plane $x_2=y_2=0$  is equivalent to a single equation of the
second order
$$
\varepsilon^2x''+\partial_x\hat V_\mu(x)=0,
$$
where ${}'$ denotes  differentiation with respect to the variable $s=\varepsilon t$ and
 $\hat V_\mu(x)$ is used to denote the formal series $\hat V_\mu(x,0)$.
 Our assumptions on the lower order terms of the Hamiltonian $\hat H_\mu$ imply that
 $$
\hat V_\mu(x)=\sum_{k,l\ge0} v_{kl}\mu^kx^l
$$
with $v_{00}=v_{01}=v_{02}=0$ and $v_{03}v_{11}\ne0$.
Multiplying the differential equation by $x'$ and integrating once we get
$$
\varepsilon^2(x')^2+2\hat V_\mu(x)+C_\varepsilon=0
$$
where $C_\varepsilon$ is a formal series in $\varepsilon^2$ (the first two terms in the sum are power series in $\varepsilon^2$ by the assumptions
of the lemma, so $C_\varepsilon$ must be in the same class).
This equation has  a unique formal solution of the form
$$
x=\sum_{n\ge1} \varepsilon^{2n}p_n(z),
\qquad
\mu=\sum_{n\ge 2}\mu_n\varepsilon^{2n},
\qquad
C_\varepsilon=\sum_{n\ge6}C_n\varepsilon^{2n},
$$
where $p_n$ is polynomial of order $n$ in $z=\frac1{\cosh^2\frac s2}$.
Indeed, differentiating $x$ with respect to $s$ we get
$$
x'=\sum_{n\ge1}\varepsilon^{2n}p_n'z'
$$
and taking the square
$$
(x')^2=\sum_{n_1,n_2\ge1}\varepsilon^{2n_1+2n_2}p'_{n_1}p'_{n_2}z^2(1-z)
$$
where we used the identity $(z')^2=z^2-z^3$.
Substituting the formal series $x$ and $\mu$ we get
$$
\hat V_\mu(x)=
\sum_{k,l\ge0}v_{kl}
\sum_{\substack{\bm j\ge2,
\, \bm i\ge1}} \varepsilon^{2(j_1+\dots j_k+i_1+\dots+i_l)}\mu_{j_1}\dots\mu_{j_k}p_{i_1}\dots p_{i_l}\,.
$$
After substituting  these expressions into the equation we get
$$
\sum_{n_1,n_2\ge1}\varepsilon^{2n_1+2n_2+2}p'_{n_1}p'_{n_2}(z^2-z^3)+2
\sum_{k,l\ge0}v_{kl}\sum_{\bm j\ge2,\, \bm i\ge1} \varepsilon^{2(j_1+\dots j_k+i_1+\dots+i_l)}\mu_{j_1}\dots\mu_{j_k}p_{i_1}\dots p_{i_l}
+\sum_{k\ge3} \varepsilon^{2k}C_k=0.
$$
This equality is treated in the class of formal series in powers of $\varepsilon^2$.
The leading order is  of order of $\varepsilon^{6}$. Collecting all terms of this order we get
$$
(p'_{1})^2(z^2-z^3)+2v_{03}p_1^3+2v_{11}\mu_2p_1+C_3=0\,.
$$
Looking for $p_1$ in the form
$p_1=p_{10}+p_{11}z$
we get
$$
p_{11}^2(z^2-z^3)+2v_{03}(p_{10}^3+3p_{10}^2p_{11}z+3p_{10}p_{11}^2z^2+p_{11}^3z^3)+2v_{11}\mu_2(p_{10}+p_{11}z)+C_3=0.
$$
This equation is equivalent to the following system for the coefficients:
$$
-p_{11}^2+2v_{03}p_{11}^3=0,
\qquad
p_{11}^2+6v_{03}p_{10}p_{11}^2=0,
\qquad
v_{03}3p_{10}^2p_{11}+2v_{11}\mu_2p_{11}=0,
\qquad
2v_{03}p_{10}^3+C_3=0.
$$
This system has a unique solution with $p_{11}\ne0$, which leads to a non-constant $p_1$: 
$$
p_{11}=(2v_{03})^{-1},\quad p_{10}=-(6 v_{03}) ^{-1},\quad \mu_2=(24v_{03}v_{11})^{-1},\quad
C_3=-2v_{03}p_{10}^3.
$$
Then we continue by induction. Suppose that for some $n\ge2$ all coefficients are defined uniquely up to $p_{n-1}$, $\mu_{n}$
and $C_{n+1}$. Then
collect the terms of order  $\varepsilon^{2n+4}$  to obtain an equation of the form
$$
2p_{11}p_{n}'(z^2-z^3)=3v_{03}p_1^2p_n+v_{11}\mu_2p_n+v_{11}\mu_{n+1}p_1+C_{n+2}+pol_{n+2},
$$
where $pol_{n+2}$ is a polynomial of order $n+2$ in $z$ with coefficients depending on already known ones.
We can find the coefficients of $p_n$ starting from the largest power of $z$. We find
$\mu_{n+1}$ from the linear term in $z$ and $C_{n+2}$ from the constant term.
In the essence we solve a linear algebraic system with a triangle matrix
with non-vanishing elements on the diagonal. So the coefficients are unique.

\smallskip

If $H_\mu$ is an analytic family defined by\/ {\rm (\ref{Eq:mainH})} then there is an analytic coordinate change which
moves the remainder term $R_\mu$ beyond any fixed order $n$. Neglecting this remainder we obtain
a polynomial Hamiltonian of the form (\ref{Eq:HamNF}) and  it is not too difficult to verify that
for this Hamiltonian $\varepsilon=\tilde \lambda_\mu$ where $\tilde \lambda_\mu$ is
an exponent of the saddle-centre equilibrium of the truncated normal form.
Since $\lambda_\mu$ is not changed by smooth coordinate changes,
the Taylor expansions of $\tilde\lambda_\mu$ and $\lambda_\mu$ in powers of $\mu^{1/4}$ coincide in the first $n$ terms
(indeed, our formal computations show that the first $n$-terms of this series are
uniquely determined by the first $n$ orders of the Hamiltonian,  these terms are the same
as the remainder affects terms of higher order only).
\end{proof}

\subsection{Formal variational equation\label{Se:vareqf}}

In addition to the formal separatrix solution we will need to study formal solutions
for the corresponding variational equation. These formal solutions will be used
to approximate analytic solutions of the variational equation with an error being
of a sufficiently high order in $\varepsilon$.

The variational equation near the formal solution provided by Lemma~\ref{Le:fsep} has the form
\begin{equation}\label{Eq:formalvar}
\begin{array}{ll}
\dot x_1=y_1,\quad
&\dot x_2=\partial_2\hat V_\mu(\hat x_1,0)y_2,\\[8pt]
\dot y_1=-\partial_1\hat V_\mu(\hat x_1,0)x_1,\quad
&\dot y_2=-\partial_2\hat V_\mu(\hat x_1,0)x_2.
\end{array}
\end{equation}
This system is split into two independent pairs of  linear equations.
This property allows us to solve this system  explicitly.
Indeed, a direct substitution shows that the function
$$
\hat{\bm\xi}=(0,0,1,i)\exp\left(i\int_0^t\partial_2\hat V_\mu(\hat x_1,0)\right)
$$
satisfies the system (\ref{Eq:formalvar}). In order to give precise meaning to this
expression in the class of formal series we define the first formal solution
of the formal variational equation by
\begin{equation}\label{Eq:xif}
\hat{\bm\xi}_2=(0,0,1,i)\exp i\hat u_\mu
\end{equation}
where the formal series $\hat u_\mu$ is defined by termwise integration
\begin{equation}\label{Eq:uf}
\hat u_\mu:=\int_0^t\partial_2\hat V_\mu(\hat x_1,0)=\omega_\mu t+\sum_{k\ge0}\varepsilon^{2k+1}u_k(\varepsilon t)
\end{equation}
where
\begin{equation}\label{Eq:uk}
u_k(s)=\sum_{l=0}^{k}u_{k,l}\frac{\sinh \tfrac s2}{\cosh^{2l+1}\tfrac s2},
\qquad
\omega_\mu=\sum_{k\ge0}\omega_k\varepsilon^{2k}\,,
\end{equation}
and $u_{k,l}$, $\omega_k$ are  real coefficients.
These coefficients can be computed by substituting the
formal series $\hat x_1$ into $\partial_1 \hat V_\mu(x_1,0)$
and integrating the result  termwise.%
\footnote{%
In order to check that the result of integration has the stated form consider
 $w_k=\int_0^s\frac{ds'}{\cosh^{2k}\frac {s'}2}$. Then $w_1=\tanh \frac s2$ and integrating by parts we get
$$
w_{k}=\frac{2}{2k-1}\left((k-1)w_{k-1}+\frac{\sinh\frac s2}{\cosh^{2k-1}\frac s2}\right)\,.
$$
}

The other solution $\hat{\bm \xi}_1$ is defined using the real symmetry:
$\hat{\bm \xi}_1(t)=\overline{\hat{\bm \xi}_2(\overline{t})}$.

The third solution is obtained by differentiating the formal separatrix with respect to $t$:
$$
\hat{\bm \xi}_3=(\dot{\hat x}_1,\dot{\hat y}_1,0,0)\,.
$$
Finally, the fourth solution is found  from the normalisation assumptions
$$
\Omega(\hat {\bm \xi}_1,\hat {\bm \xi}_4)=\Omega(\hat {\bm \xi}_2,\hat {\bm \xi}_4)=0,
\quad
\Omega(\hat {\bm \xi}_3,\hat {\bm \xi}_4)=1.
$$
We sketch the derivation for its form.
A solution to the  equation
$$
\left|\begin{array}{ll}\hat x_1'&\hat v\\ \hat x_1''&\hat v'\end{array}\right|=1
$$
also satisfies the same variational equation as $\hat x_1'$.
Looking for a formal solution of the form
$$
\hat v=\sum_{k\ge0}\varepsilon^{2k}v_k
$$
with
\begin{equation}\label{Eq:forvare4}
v_0=c_0 sz'+z^{-1}d_2(z)\qquad\mbox{and}\qquad
v_k=c_k(z) sz'+z^{-1}d_{k+2}(z)
\end{equation}
where $d_k$, $c_k$ are polynomial 
in $z=\frac1{\cosh^2\frac s2}$ (of the order indicated by the subscript).
Substituting the series into the equation and collecting  
the leading terms in $\varepsilon$, we get 
$$
z' v_0'-z''v_0=p_{11}^{-1}.
$$
This equation can be solved explicitly:
$$
v_0=\left(
-1/2 z^{-1} + 15/8 ( s z' + 2  z - 2/3 )\right)p^{-1}_{11}\,.
$$
The solution of the equation is defined up to adding a multiple of $z'$,
which satisfies the corresponding homogeneous equation,
but only one of those solutions has the desired form (\ref{Eq:forvare4}).
Thus  our expression for $v_0$ provides the unique solution in 
our class.

Acting in a similar way and collecting terms of  order  $\varepsilon^{2n}$
we get an equation for $v_{n}$ with $n\ge1$:
$$
\sum_{0\le k\le n}\bigl( p_{n-k}'v_k'-p_{n-k}''v_k \bigr)=0\,,
$$
where $p_k(s)$ are defined by equation~(\ref{Eq:pkqk}).
An induction in $n$ can be used to prove that this equation
uniquely determines a function $v_n$ of the form (\ref{Eq:forvare4})
provided $v_k$ with $k<n$ already have that form.
The induction step is rather straightforward: we plug (\ref{Eq:forvare4}) into the equation
and take into account that $p_k$ is a polynomial in $z$ of order $k$,
then $c_n'$ is determined by the terms proportional to~$s$. The constant term of $c_n$
is used to satisfy a compatibility condition in a system for coefficients of $d_{n+2}$.
The identities $(z')^2=z^2-z^3$, $z''=z-\frac32 z^2$ are used to simplify equations.

\subsection{Re-expansion near the singularity at {\protect $s=i\pi$}.\label{Se:reexp}}

It is easy to see that all coefficients of the series (\ref{Eq:fomsep}) converge to a constant both for $t\to+\infty$
and for $t\to-\infty$. So any partial sum of the series represents a closed loop on the plane $(x_1,y_1)$.
We say that the series represent a formal (or ghost) separatrix as the series are expected to diverge
generically.
It can be shown that the same series provide an asymptotic expansions both
for the stable and unstable solution of the original system (after reversing the transformation
which brings the original Hamiltonian to the formal normal form).
Consequently, the series (\ref{Eq:fomsep}) do not directly distinguish between the stable and unstable solution.

The coefficients of the series  (\ref{Eq:fomsep}) and (\ref{Eq:uf}) are meromorphic functions with poles
at  $s_k=i\pi+2\pi i k$ for all $k\in\mathbb Z$. The coefficients are $2\pi i$ periodic in the variable $s$,
so the following arguments equally apply to every $s_k$.

Since  $\cosh(\frac s2)$ has a simple zero
at $i\pi$, equations (\ref{Eq:pkqk}) and (\ref{Eq:uf}) imply that $p_k$, $q_k$ and $u_k$
have poles of order $2k+2$, $2k+3$ and $2k+1$ respectively. Since
$\cosh\frac{s+i\pi}{2} =i\sinh\frac{s}{2}$ and $\sinh\frac{s+i\pi}{2} =i\cosh\frac{s}{2}$,
equations (\ref{Eq:pkqk}) and (\ref{Eq:uk})
all coefficients of the Laurent expansions for $p_k$, $q_k$ and $u_k$ are real
and  contain either only even powers of $(s-i\pi)$ for $p_k$
or  only odd powers for $q_k$ and $u_k$.
Substituting the Laurent  expansions into the series
(\ref{Eq:fomsep}) and (\ref{Eq:uf}) we obtain new formal series of the form
$$
\hat X_1=\sum_{k\ge 1,l\ge -k}\varepsilon^{2k}(\varepsilon t-i\pi)^{2l}\tilde p_{k,l},\quad
\hat Y_1=\sum_{k\ge 1,l\ge -k}\varepsilon^{2k+1}(\varepsilon t-i\pi)^{2l-1}\tilde q_{k,l},\quad
\hat X_2=\hat Y_2=0,
$$
and
$$
\hat U=\omega_\mu t+\sum_{k\ge 0,l\ge -k}\varepsilon^{2k+1}(\varepsilon t-i\pi)^{2l-1}\tilde u_{k,l}\,,
$$
where $\tilde p_{k,l}$, $\tilde q_{k,l}$ and $\tilde u_{k,l}$ are real.
Note that these series contain both positive and negative powers of $(\varepsilon t-i\pi)$.
It is convenient to shift the origin of the time variable into the centre of the Laurent expansion
and introduce a new time variable
\begin{equation}\label{Eq:shiftedtime}
\tau=t-i\frac{\pi}{\varepsilon}\,.
\end{equation}
Substituting this time in the series above we get
$$
\hat X_1=\sum_{k\ge 1,l\ge -k}\varepsilon^{2k+2l}\tau^{2l}\tilde p_{k,l},
\qquad
\hat Y_1=\sum_{k\ge 1,l\ge -k}\varepsilon^{2k+2l}\tau^{2l-1}\tilde q_{k,l}
$$
and
$$
\hat U
=i\pi\omega_\mu \varepsilon^{-1}+\omega_\mu \tau+\sum_{k\ge 0,l\ge -k}\varepsilon^{2k+2l}\tau^{2l-1}\tilde u_{k,l}\,.
$$
Collecting together the coefficients which have the same order in $\varepsilon$ we get
\begin{equation}\label{Eq:hatX1Y1}
\hat X_1=\sum_{m\ge0}\varepsilon^{2m}\hat A_m,
\qquad
\hat Y_1=\sum_{m\ge0}\varepsilon^{2m}\hat B_m
\end{equation}
and
\begin{equation}\label{Eq:hatU}
\hat U=i\pi\omega_\mu \varepsilon^{-1}+\sum_{m\ge 0}\varepsilon^{2m}\hat U_m
\end{equation}
where $\hat A_m$, $\hat B_m$ and $\hat U_m$ are formal series defined by
\begin{equation}\label{Eq:AB}
\hat A_m
=\sum_{l\le m-1}\tau^{2l}\tilde p_{m-l,l}
,
\qquad
\hat B_m
=\sum_{l\le m-1}\tau^{2l-1}\tilde q_{m-l,l}
\end{equation}
and
\begin{equation}\label{Eq:hatUm}
\hat U_m
=\omega_k \tau+\sum_{l\le m}\tau^{2l-1}\tilde u_{m-l,l}\,.
\end{equation}
We note that each of these series contains a finite number of terms with positive powers of $\tau$
and an infinite formal series in $\tau^{-1}$.
In the next section we will use this series
in the process known as {\em complex matching}.

The new series (\ref{Eq:hatX1Y1}) satisfy the same formal equations 
(\ref{Eq:formalvar}) as the original series $\hat x_1$, $\hat y_1$, and 
\begin{equation}\label{Eq:eta_1}
\hat{\bm\eta}_2=(0,0,1,i)\exp\left(i\sum_{m\ge0}\varepsilon^{2m}\hat U_m\right)
\end{equation}
formally satisfies the variational equation (\ref{Eq:formalvar}) with $\hat x_1$ replaced by $\hat X_1$.
Note that in the definition of $\hat{\bm \eta}_2$ we skip the first term of (\ref{Eq:hatU}).
Consequently, partial sums of $\hat{\bm\xi}_2$ should be compared with $e^{-\pi\omega_\mu \varepsilon^{-1}}\hat{\bm \eta}_2$.

Finally, the formal series $\hat A_1$ and $\hat B_1$ are trivial,
i.e. all coefficients of these two series vanish.
Indeed, it can be checked that $(\hat A_1,\hat B_1,0,0)$ satisfies a homogeneous
variational equation around $\hat {\bm X}_0=(\hat A_0,\hat B_0,0,0)$.
Arguments of section~\ref{Se:vareq} can be modified to show that
any nontrivial formal solution in the class of formal power series in $\tau^{-1}$
is proportional to $\dot{\hat{\bm X}}_0$. The first component of  $\dot{\hat {\bm X}}_0$ contains odd powers of $\tau^{-1}$ only
while $\hat A_1$ contains  only even orders. Consequently $\hat A_1\equiv0$. Analysis of the second component
shows $\hat B_1\equiv0$.

\section{Approximations for the separatrix solutions and the complex matching\label{Se:approx}}

In section~\ref{Se:formal} we derived various formal solutions to the normal form equation.
Partial sums of the formal solutions (i.e. truncated formal series) satisfy the original analytic equations up to
a small remainder. Then it can be shown that they provide rather accurate approximations
for the corresponding analytic solutions both for the Hamiltonian equation and the variational one.
The formal series are asymptotic, i.e., the difference between
a truncated series and the corresponding analytic solution is of the order of the first skipped term.
A typical series used in this analysis is a power series in $\varepsilon$ with
time-dependent coefficients. The errors of the approximations depend on time, so the
asymptotics are not uniform. Therefore it is  important
to pay attention to domains of validity of the asymptotic expansions
on the complex plane of the time variable.

Consider the system of equations (\ref{Eq:mainHamEq}) associated with the Hamiltonian function (\ref{Eq:mainH})
assuming that the remainder of the normal form is of order $2N+3$ for some integer $N\ge 2$.
Then we can use the first $2N$ terms of the series (\ref{Eq:fomsep}) to approximate the stable and unstable solutions:
\begin{equation}\label{Eq:approx1}
\bm x^\pm_\mu(t)=\left(\sum_{k=1}^{2N}\varepsilon^{2k}p_k(\varepsilon t),\sum_{k=1}^{2N}\varepsilon^{2k+1}q_k(\varepsilon t)
,0,0\right)+\bm r_N^\pm(t,\varepsilon)
\end{equation}
where $\varepsilon=\lambda_\mu$.
It can be shown that $\bm r_N^-(t,\varepsilon)=O(\varepsilon^{4N+1})$ for $t\le0$
and $\bm r_N^+(t,\varepsilon)=O(\varepsilon^{2N+1})$ for $t\ge0$.
This estimates can be extended onto complex values of $t$ such that
the slow time $s=\varepsilon t$  avoids small neighbourhoods
of the singular points  at $s=\pm i\pi$.
A more accurate estimate (similar to one used in \cite{G2000})
shows that the series retain the asymptotic property at least up to
$| t-i\frac{\pi}{\varepsilon}|=\varepsilon^{-1/2}$ (with $\Re t\le 0$ for the unstable solution
and $\Re t\ge0$ for the stable one). For these values of $t$ the error term
becomes notably larger but still  small:
$\bm r_N^\pm(t,\varepsilon)=O\left(\varepsilon^{4N+1}| t-i\frac{\pi}{\varepsilon}|^{4N}\right)=O(\varepsilon^{2N+1})$.
On the other hand, in this region $|s-i\pi|=\varepsilon^{1/2}$ is small
so the functions $p_k$ and $q_k$ can be replaced by partial sums of their Laurent
expansions (\ref{Eq:hatX1Y1}), (\ref{Eq:AB}) described in section \ref{Se:reexp} without increasing the order of the error term.

In order to estimate $\bm x^\pm_\mu(t)$ for values of $t$ closer to the singularity
 we use a different approximation:
\begin{equation}\label{Eq:approx2}
\bm x^\pm_\mu(t)=\sum_{m=0}^{N-1}\varepsilon^{2m}\bm X^\pm_m(\tau)+\bm R_N^\pm(\tau)\,,
\end{equation}
where the time variables $t$ and $\tau$ are related by equation (\ref{Eq:shiftedtime}).
The coefficients $\bm X^\pm_m(\tau)$ are defined via the process known as {\em
complex matching}. The method uses partial sums of the series (\ref{Eq:AB})
as asymptotic condition for $\tau\to\infty$ in the definition of  $\bm X^\pm_m(\tau)$.
The number of retained terms is chosen in such a way that
the comparison of the approximations
(\ref{Eq:approx1}) and (\ref{Eq:approx2})
leads to the estimate $\bm R_N^\pm(\tau)=\bm r_N^\pm(t,\varepsilon)+O(\varepsilon^{2N})=O(\varepsilon^{2N})$
for $|\tau|=\varepsilon^{-1/2}$
(with $\Re \tau\le 0$ for the unstable solution
and $\Re \tau\ge0$ for the stable one).

In order to derive equations for  $\bm X^\pm_m$ we note that
Lemma~\ref{Le:fsep} implies  $\mu=\sum_{k=2}^\infty\mu_k\varepsilon^{2k}$.
Since this  series converges  we can express the Hamiltonian in the form
of a convergent series 
\begin{equation}
\label{Eq:Heps}
H_\mu=\sum_{k=0}^\infty\varepsilon^{2k}H_k.
\end{equation}
Obviously, the leading term of this series coincides with $H_\mu$ with $\mu=0$
and $H_1\equiv0$. Indeed,   $\mu$ is of the order of $\varepsilon^4$ and,
consequently, the series do not include  terms proportional to $\varepsilon^2$.
Plugging   (\ref{Eq:Heps}) and (\ref{Eq:approx2})
into  (\ref{Eq:mainHamEq}) and collecting terms of equal order in $\varepsilon^2$
we obtain a system of equations for the coefficients $\bm X_m^\pm$. For example for $m=0,1$ and $2$ we obtain
\begin{eqnarray}
\label{Eq:ham0}
\dot {\bm X}^\pm_0&=&\mathrm{J}H_0'(\bm X^\pm_0),
\\
\dot {\bm X}^\pm_1&=&\mathrm{J}H_0''(\bm X^\pm_0)\bm X^\pm_1,\\
\dot {\bm X}^\pm_2&=&\mathrm{J}H_0''(\bm X^\pm_0)\bm X^\pm_2+\mathrm{J}H_2'(\bm X^\pm_0).
\end{eqnarray}
The equation for $\bm X^\pm_0$ is non-linear. The equation for $\bm X^\pm_1$ is linear
and homogeneous. The equations for $\bm X^\pm_m$ with $m\ge2$ are all linear and non-homogeneous.

Let $ \bm X_m^{(N)}(\tau)$ be the sum of the first $N$ terms of the series
 $(\hat A_m,\hat B_m,0,0)$ where $\hat A_m$, $\hat B_m$ are defined in (\ref{Eq:AB}).
 Since the leading orders of the Hamiltonian $H_\mu$ coincide with the normal form,
 $ \bm X_m^{(N)}(\tau)$ approximately satisfy the corresponding equation up to a
 remainder determined by the order of the first skipped term.
It can be shown that these equations have unique analytic solutions such that $\bm X_m^\pm(\tau)\sim \bm X_m^{(N)}(\tau)$
as $| \tau|\to\infty$ in a sector $D^\pm$ respectively.
In other words, either $\tau$ (for the unstable solution) or $-\tau$ (for the stable one)
is in
$$D^-=\left\{\tau\in\mathbb C: |\arg(\tau)|>\frac\pi4\right\}.$$
We also define $D^+=-D^-$.
The asymptotic assumption implies $\bm X_1\equiv0$ as the formal series $\hat A_1=\hat B_1=0$.

\medskip

It can be shown that the upper bound on the remainders in (\ref{Eq:approx2}),
$$
\bm R_N^\pm(\tau)=O(\varepsilon^{2N})
$$
can be extended to a neighbourhood of the segment of the imaginary axis defined by
$\Re\tau=0$, $-\varepsilon^{-1/2}\le\Im\tau\le -c$,
where $c$ is a  positive constant.

\medskip

We also need an approximation for the function $\bm \xi_2$ which is defined
in section~\ref{Se:vareq} as a solution of the variational equation
 (\ref{Eq:mainvareq}) around
the unstable solution $\bm x^-_\mu$. It is convenient to construct
the approximation for the function $\bm\eta_2$ defined by
the equality
\begin{equation}\label{Eq:trans}
\bm \xi_2=e^{-\pi\omega_\mu\varepsilon^{-1}}\bm \eta_2^-
\end{equation}
where the pre-factor is related to the constant term in (\ref{Eq:hatU}).
Since the variational equation is linear, the function $\bm \eta_2$ also  satisfies
 (\ref{Eq:mainvareq}). Following the procedure of complex matching
 we look for a representation of $\bm\eta_2$ in the form
\begin{equation}\label{Eq:eta}
{\bm \eta}_2^-=\sum_{k=0}^{N-1}\varepsilon^{2k}\bm \eta_k(\tau)+\bm \rho_N(\tau,\varepsilon),
\end{equation}
where the coefficients satisfy
\begin{eqnarray}
\label{Eq:eta0eq}
\dot{\bm \eta}_0&=&\mathrm{J}H_0''(X^-_0)\bm\eta_0,\\
\dot{\bm \eta}_1&=&\mathrm{J}H_0''(X^-_0)\bm\eta_1,\\
\dot{\bm \eta}_2&=&\mathrm{J}H_0''(X^-_0)\bm\eta_2+\mathrm{J}H_1''(\bm X^-_0)\bm \eta_0
\\
&&\dots\nonumber
\end{eqnarray}
The solutions to these equations are also chosen to satisfy
 asymptotic conditions obtained by complex matching,
 based on  comparison with the formal solution (\ref {Eq:eta_1}).
In particular, the asymptotic condition for $\bm\eta_0$ has the form
\begin{equation}\label{Eq:eta0}
\bm\eta_0(\tau)\sim(0,0,1,i)e^{i\hat U_0^{(N)}(\tau)}
\end{equation}
as $\tau\to\infty$ in $D^-$. In this equation  $\hat U_0^{(N)}$
denotes the sum of the first $N$ terms of the series (\ref{Eq:hatUm})
with $m=0$.

In order to complete the proof, one should demonstrate existence of the coefficients and
upper bounds for the remainder terms. The corresponding arguments are similar
in all statements of this section: the differential equation is rewritten as an integral
equation for a reminder term. Then a contraction mapping arguments
are used to  bound the reminder in a properly chosen Banach space.

\section{Stokes constants}
Let $\bm X_0^-$ be the unstable solution of the Hamiltonian equation (\ref{Eq:ham0}).
This solutions converges to zero as $\tau\to\infty$ in the sector $D^-$ and consequently
represents the unstable manifold associated with the non-hyperbolic equilibrium
of the Hamiltonian
$$
H_0=\frac{y_1^2}2+V_0(x_1,I)+R_0(x_1,y_1,x_2,y_2).
$$
The stable solution $\bm X_0^+$ represents
the stable manifold of the same equilibrium. It satisfies the same Hamiltonian equation
and has the same asymptotic expansion as $\bm X_0^-$ but in a different
domain, namely in $D^+=-D^-$ (i.e., $\tau\in D^+$ iff $-\tau\in D^-$).
It follows immediately that
$$
\bm\delta_0(\tau)=\bm X_0^+(\tau)-\bm X_0^-(\tau)
$$
converges to zero faster than any power of $\tau^{-1}$ when $\tau$ converges to infinity in $D^-\cap D^+$.
In order to describe the difference between the stable and unstable solution we define a
Stokes constant
\begin{equation}\label{Eq:a0}
b_0=
\lim_{\Im\tau\to-\infty}
\Omega(\bm \delta_0(\tau),\bm\eta_0(\tau))
\end{equation}
where $\bm\eta_0$ is an analytic solution of the variational equation (\ref{Eq:eta0eq})
which satisfies the asymptotic condition (\ref{Eq:eta0}) at infinity in the sector $D^-$. 
Since $\bm\delta_0$ is a difference of two solutions of the same equation and 
small, it satisfies the variational equation with a error proportional to $|\bm\delta_0|^2$.
If $\bm\delta_0$ satisfied the variational equation exactly, the symplectic form in (\ref{Eq:a0})
would be constant. 
Note that while $\bm \delta_0(\tau)$ exponentially decays along the negative imaginary
semi-axis, the function $\bm\eta_0(\tau)$ grows exponentially,
these two tendencies compensate each other and,
using the exponential decay of $\bm\delta_0(\tau)$,
it can be shown that the 
convergence to the limit in (\ref{Eq:a0}) is exponentially fast.

More generally, we define
\begin{equation}\label{Eq:bn}
b_n=
\lim_{\Im\tau\to-\infty}
\sum_{m=0}^n\Omega(\bm\delta_m(\tau),\bm\eta_{n-m}(\tau))
\end{equation}
where
\begin{equation}\label{Eq:deltak}
\bm\delta_m(\tau)=\bm X_m^+(\tau)-\bm X_m^-(\tau),
\end{equation}
$\bm X^\pm_k$ and $\bm\eta_k$ are defined in Section~\ref{Se:approx}.
It can be shown that the expression under the limit converges to the limit exponentially:
\begin{equation}\label{Eq:limits}
\left|b_n-\sum_{m=0}^n\Omega(\bm\delta_m(\tau),\bm\eta_{n-m}(\tau))\right|\le \mathrm{const}_n |\tau|^{k(n)}e^{\omega_0\Im \tau}
\end{equation}
when $\Im\tau\to-\infty$.
We do not know explicit expressions for the limits in (\ref{Eq:a0}) and (\ref{Eq:bn}).
For an explicitly given Hamiltonian function $H_0$ the value of $b_0$
can be found numerically relatively easy with the help of the methods
described in \cite{GL01}. The method relies on computation
of formal solutions both for the Hamiltonian and the corresponding
variational equations. The finite parts of the formal series are used
to obtain an accurate approximation for $\bm X_0^\pm(\tau)$ for large $\Re\tau$.
Then values of $\bm X_0^\pm(\tau)$ on the imaginary axis are obtained by
numerical integration which keeps $\Im\tau=\mathrm{const}$.
Then $\bm\delta_0$ is evaluated and substituted into (\ref{Eq:a0}).
The function $\bm\eta_0$ can be approximated on the basis of (\ref{Eq:eta0}).
We note that this computation can be performed
in the original coordinates. In this case the explicit knowledge
of the transformation to the normal form is not needed.
On the other hand, the normal form theory still plays an important role
as it  is used to chose the correct Ansatz for the formal solutions
which replace the formal series of section~\ref{Se:formal}.

\section{Derivation of the asymptotic formula}
In this section we complete the proof of the main theorem using the estimates
from the previous two sections to approximate the function $\theta_1$
defined by equation (\ref{Eq:theta1}). First we note
$$
\theta_1(t)=
\frac1{2i}\Omega(\bm\delta_\mu(t),\bm\xi_2(t))=
\frac{e^{-\pi\omega_\mu\varepsilon^{-1}}}{2i}\Omega(\bm\delta_\mu(t),\bm \eta_2^-(\tau)).
$$
where the exponential factor appears due to the shift of the origin in the time variable
(see equations (\ref{Eq:shiftedtime}) and (\ref{Eq:trans})). This prefactor
permits us to obtain an approximation for the exponentially
small $\theta_1$ from estimates of values
of the symplectic form which have errors of order of $\varepsilon^N$ only.
The estimate (\ref{Eq:approx2}) implies that
$$
\bm\delta_\mu(t)=\bm x_\mu^+(t)-\bm x_\mu^-(t)
=\sum_{n=0}^{N-1}\varepsilon^{2n}\bm\delta_k(\tau)+(\bm R_N^+(\tau,\varepsilon)-\bm R_N^-(\tau,\varepsilon)),
$$
where $\bm\delta_k$ are defined by (\ref{Eq:deltak}).
Substituting this expression and (\ref{Eq:eta}) into the definition of the function $\theta_1$
we obtain that
$$
\theta_1(t)=
\frac{e^{-\pi\omega_\mu\varepsilon^{-1}}}{2i}
\left(\sum_{n=0}^N\varepsilon^{2n}
\sum_{k=0}^n\Omega(\bm\delta_{k}(\tau),\bm\eta_{n-k}(\tau))+Q_N(\varepsilon,\tau)\right)
$$
where $Q_N$ is the collection of all terms not explicitly included into the  sums.
Taking into account the definitions of $b_n$ and the estimates (\ref{Eq:limits})
we obtain for $\Im\tau=-\omega_0^{-1}\log\varepsilon^{-N}$
$$
\theta_1(t)=
\frac{e^{-\pi\omega_\mu\varepsilon^{-1}}}{2i}\left(\sum_{n=0}^N\varepsilon^{2n}b_n+O_N\right)\,.
$$
Finally, multiplying by the complex conjugate we get
$$
|\theta_1|^2=
\frac{e^{-2\pi\omega_\mu\varepsilon^{-1}}}{4}
\left(\sum_{n=0}^N\varepsilon^{2n}\sum_{k=0}^nb_k\bar b_{n-k}+O_N\right).
$$
Setting
\begin{equation}\label{Eq:an}
a_n=\frac12\sum_{k=0}^nb_k\bar b_{n-k}
\end{equation}
and using (\ref{Eq:EellipGen}) we obtain the desired estimate (\ref{Eq:mainasymp}).


\section{Comparison with the Melnikov method\label{Se:melnikov}}

Melnikov method is a useful tool in studying the splitting of separatrices.
In the situation when the separatrix splitting is exponentially small compared
to a natural parameter of the problem Melnikov method generically fails.
Although there are classes of system where the usage of the Melnikov method
can be justified (see e.g. \cite{GL01,BaldomaSeara}).

\subsection{Melnikov method and the splitting of the separatrix loop}

In parallel with the  Hamiltonian defined by (\ref{Eq:mainH})
we can consider a two-parameter analytic family
$$
H_{\mu,\nu}=\frac{y_1}2+V_\mu(x_1,I)+\nu R_\mu(x_1,y_1,x_2,y_2).
$$
This family depends on an additional parameter $\nu$
and interpolates between $H_\mu$ and an integrable normal form.
Moreover, for every fixed $\nu=\nu_0$, the one-parameter family 
$H_{\mu,\nu_0}$ satisfies the assumptions of this paper.
It is easy to see that the Hamiltonian $H_{\mu,0}$ is integrable
and has a separatrix loop on the plane $I=0$ for every $\mu\in(0,\mu_0)$.
The Melnikov method can be used to study the splitting of this loop
for small $\nu$.

It is convenient to translate the saddle-centre equilibrium $\bm p_{\mu,\nu}$ into the origin.
Assuming this step is already performed, we noticeably simplify the discussion
of convergence of integrals used later in this section.
It is convenient to define the complex variable $z_2=x_2+i y_2$.
Then the Hamiltonian equations imply that
$$
\dot z_2=-i \partial_2 V_{\mu,\nu}(x_1,I)z_2+\nu(\partial_{y_2}R_\mu-i\partial_{x_2}R_\mu)\,.
$$
A vanishing at $-\infty$ (resp. $+\infty$) solution of this equation must satisfy
the following integral equation
$$
z_2^\pm(t)=\int_{\pm\infty}^te^{-i \int_s^t\partial_2 V_{\mu,\nu}(x_1^\pm,0)}
\left(-i (\partial_2 V_{\mu,\nu}(x_1^\pm,I)-\partial_2 V_{\mu,\nu}(x_1^\pm,0))z_2^\pm
+\nu(\partial_{y_2}R_\mu-i\partial_{x_2}R_\mu)\right)ds
$$
where the integral is taken along the stable (reps. unstable) trajectory $\bm x_{\mu,\nu}^\pm=(x_1^\pm,y_1^\pm,x_2^\pm,y_2^\pm)$.
Note that no approximation is involved at this stage.

For $\nu=0$ the separatrix loop is located on the invariant plane $x_2=y_2=0$ and we have $z_2^\pm\equiv0$.
The smooth dependence of the solutions on $\nu$
implies
$|z_2^\pm(t)|=O(\nu)$ and $I(x_2^\pm,y_2^\pm)=|z_2^\pm|^2/2=O(\nu^2)$
for the values of $t$ used in our integrals.
So we can bound the integral of the first term by $O(\nu^2)$.
Then we use $x^\pm_{\mu,\nu}=x^0_\mu+O(\nu)$ to get
$$
z_2^\pm(0)=\nu\int_{\pm\infty}^0
e^{i \int_0^s\partial_2 V_{\mu,\nu}(x^0_\mu,0)}(\partial_{y_2}R_\mu-i\partial_{x_2}R_\mu)ds
+O(\nu^2),
$$
where the integral is taken along the separatrix solution
$x_\mu^0$  of the
integrable Hamiltonian system defined by
$$
H_{\mu,0}=\frac{y_1^2}2+V_{\mu}(x_1,I).
$$
We define a Melnikov integral:
$$
M_\mu=i
\int_{-\infty}^\infty \left.
e^{i\int_0^s \partial_2 V_{\mu}(x^0_\mu,0)}
\left(
\frac{\partial R_\mu}{\partial x_2}
-
i\frac{\partial R_\mu}{\partial y_2}
\right)\right|_{(x^0_\mu(\varepsilon t),y^0_\mu(\varepsilon t),0,0)}dt\,.
$$
Then we get
$$
z_2^+(0)-z_2^-(0)=\nu M_\mu+O(\nu^2).
$$
The Melnikov integral $M_\mu$ depends on $\mu$. As it is a rapidly oscillating integral
of an analytic function, it is exponentially small compared to $\mu$.

\medskip

In particular, we  can start the procedure for a Hamiltonian
which is transformed to the normal form up to the third order only,
so
$$
V_\mu(x,I)=\omega_0I-\mu x+\frac{ x^3}3.
$$
Note that to avoid unessential constant in the estimates we normalised the constants so $a=b=1$ in (\ref{Eq:potential}),
which can be achieved without restricting the generality. Then $x^0_\mu(s)=(\varepsilon^2 z(s),\varepsilon^3z'(s),0,0)$.
The function $z(s)$ has second order poles at $s_k^*=i\pi+2\pi i k$, $k\in\mathbb Z$.
If additionally $R_\mu$ is polynomial in $x$ (this assumption restricts the class of Hamiltonian systems we consider),
then the Melnikov integral is computed explicitly using residues:
$$
M_\mu=
\frac{-2\pi \varepsilon^{-1}}{e^{\pi\omega_\mu\varepsilon^{-1}}+e^{-\pi\omega_\mu\varepsilon^{-1}}}
{\mathrm{Res}}
\left(
e^{-i \omega_\mu\varepsilon^{-1}s}
\left.\left(
\frac{\partial R_\mu}{\partial x_2}
-
i\frac{\partial R_\mu}{\partial y_2}
\right)\right|_{(x^0_\mu(s+i\pi),y^0_\mu(s+i\pi),0,0)},s=0
\right)
$$
For example, let $R_\mu=c_0x_1^m x_2$ and $x^0_\mu=\varepsilon^2 z(s)$, then
$$
M_\mu=\frac{2\pi i\varepsilon^{-1}}{e^{\pi\omega_\mu\varepsilon^{-1}}+e^{-\pi\omega_\mu\varepsilon^{-1}}}
{\mathrm{Res}}\left(
e^{-i \omega_\mu\varepsilon^{-1}s} \varepsilon^{2m}c_0z^m(s+i\pi)
,s=0
\right)\
=
\frac{2\pi (-1)^{m-1}c_0}{e^{\pi\omega_\mu\varepsilon^{-1}}+e^{-\pi\omega_\mu\varepsilon^{-1}}}
\frac{\omega_\mu^{2m-1}2^{2m}}{(2m-1)!}.
$$
The Melnikov methods provides an estimate which is in a good agreement with our main theorem, but
it does not show that the error term is negligible in the case when $\nu>0$ is fixed and the bifurcation parameter
$\mu$ approaches zero.

\subsection{Melnikov method and  Stokes constants}

Melnikov method can be applied to find a derivative of the Stokes constant.
Indeed, consider an one-parameter family of the form
$$
H_{0,\nu}=\frac{y_1^2}2+V_0(x,I)+\nu R_0(x_1,y_1,x_2,y_2)
$$
where the remainder $R_0$ has a Taylor expansion without terms of the order 3 or less
and
\begin{equation}\label{Eq:potential0}
V_0(x_1,I)=\omega_0I+b\frac{x_1^3}{3}+c x_1 I+O_4.
\end{equation}
Without loosing in generality and for greater convenience we
assume $\omega_0,b>0$. For every fixed $\nu$ we can define
the Stockes constant $b_0(\nu)$ using equation (\ref{Eq:a0}).
The Melnikov method can be used to evaluate $b'(0)$.

As an example consider the Hamiltonian
$$
H_{0,\nu}=\frac{y_1^2}2+\omega_0 I-x_1^3+\nu R_0(x_1,y_1,x_2,y_2).
$$
For $\nu=0$ this Hamiltonian is integrable and its ``separatrix" solution is given
explicitly by
$$\bm
x_0=(2 t^{-2},-4t^{-3},0,0).
$$
Introduce coordinates to diagonalise the linear part by
$$
z_2=\frac{x_2+iy_2}{\sqrt2},\qquad \bar z_2=\frac{x_2-iy_2}{\sqrt2}.
$$
Note that $\bar z_2$ is a complex conjugate of $z_2$  for real $x,y$ only.
Nevertheless, since the Hamiltonian $H_{0,\nu}$ is real-analytic we can
use its real-symmetry to compare $z_2$ and $\bar z_2$ components of the
solutions.
The Hamiltonian equations imply that
$$
\dot z_2=-i\omega_0z_2-i\nu\partial_{\bar z_2}R_0
$$
and we get
$$
z_2^\pm(t)=-i\nu\int_{\pm\infty}^t e^{-i\omega_0(t-s)}\partial_{\bar z_2}R_0ds
$$
where the integral is taken over the stable/unstable solution of the Hamiltonian system.
The integrals are taken over horizontal lines in the complex plane of the variable $s$
defined by  $\Im s=\Im t$. Consider
$$
\delta_0(t)=z_2^+(t)-z_2^-(t)=
i\nu\int_t^\infty e^{-i\omega_0(t-s)}\partial_{\bar z_2}R_0^+ds
+
i\nu\int_{-\infty}^t e^{-i\omega_0(t-s)}\partial_{\bar z_2}R_0^-ds
$$
where the superscripts on $\partial_{\bar z_2}R_0^\pm$ are used to indicate that the argument
of $\partial_{\bar z_2}R_0$ is replaced by the stable (resp. unstable) solution
and let
$$
b_0(\nu)=\lim_{\Im t\to-\infty}e^{i\omega_0t}\delta_0(t).
$$
Assuming that the limit is uniform in $\nu$ and that we can swap integration and differentiation,
and using the fact that both stable and unstable solutions converge to $\bm x_0$ 
in the limit $\nu\to0$,
we differentiate the right hand side of the definition of $b_0$ with respect to $\nu$ at $\nu=0$
and get
$$
b_0'(0)
=\lim_{\Im t\to-\infty}
\int_{i\Im t-\infty}^{i\Im t+\infty} e^{i\omega_0s}\partial_{\bar z_2}R_0(2s^{-2},-4s^{-3},0,0)ds.
$$
We have obtained an expression for $b_0'(0)$ in terms of a Melnikov integral.
If $R_0$ is polynomial,
the integral is independent of $t$ for $\Im t<0$
and can be easily evaluated explicitly using the residue at $0$.

Since $b_0$ depends analytically on parameters (to be proved), then the equation above
shows that $ b_0'(0)$ does not vanish generically for polynomial $R_0$.
We can interpolate between any $R_0$ and a polynomial one. In an analytic one-parameter family
zeroes are isolated. Consequently, on an open dense set $b_0\ne0$.

We note that completing a proof for the claims of this section requires substantial amount of
additional work  to study  dependence of solutions on additional parameters
and establishing uniform convergence which lies beyond the narrow aim of this paper.

\section*{Acknowledgements}
VG's research was supported by EPRC (grant EP/J003948/1) and by the Leverhulme Trust research project.
LL was supported by RFBR (grant 14-01-00344).
A part of this project was supported by the Russian Science Foundation (grant 14-41-00044).

The authors thank Lara el Sabbagh for reading a draft of the paper,
and Alain Champneys and Carles Sim\'o for helpful comments.

\end{document}